\documentclass
{amsart}
\usepackage{graphicx}
\newtheorem{thm}{Theorem}[section]
\theoremstyle{definition}
\newtheorem{cor}[thm]{Corollary}
\newtheorem{prop}[thm]{Proposition}
\newtheorem{defn}[thm]{Definition}
\newtheorem{lem}[thm]{Lemma}

\newtheorem{rem}[thm]{Remark}

\newtheorem{ex}[thm]{Example}

\numberwithin{equation}{section}
\begin{document}

\title[The dual notion of morphic modules over commutative rings]{The dual notion of morphic modules over commutative rings \footnote{This article was accepted for publication in Journal of Mathematical Sciences.}}

\author[Faranak Farshadifar]%
{Faranak Farshadifar*}

\newcommand{\acr}{\newline\indent}
\address{\llap{*\,} Department of Mathematics Education, Farhangian University, P.O. Box 14665-889, Tehran, Iran.}
\email{f.farshadifar@cfu.ac.ir}


% ----------------------------------------------------------------
\begin{abstract}
Let $R$ be a commutative ring with identity and $M$ be an $R$-module.
The purpose of this paper is to introduce and investigate the dual notion of morphic modules over a commutative ring.
\end{abstract}
\maketitle
% ----------------------------------------------------------------
\section{Introduction}
\noindent
Throughout this paper, $R$ will denote a commutative ring with
identity and $\Bbb Z$ will denote the ring of integers. Let $M$ be an $R$-module, $L$ be a submodule of $M$, and $a \in R$. Then the notation $(L:_Ma)$  is defined by $\{m \in M \, |\, ma \in L\}$.

Let $M$ be an $R$-module.
A proper submodule $N$ of $M$ is said to be \emph{completely irreducible} if $N=\bigcap _
{i \in I}N_i$, where $ \{ N_i \}_{i \in I}$ is a family of
submodules of $M$, implies that $N=N_i$ for some $i \in I$. Every submodule of $M$ is an intersection of completely irreducible submodules of $M$ \cite{FHo06}. $M$ satisfies the \emph{double annihilator
conditions} ($DAC$ for short) if, for each ideal $I$ of $R$,
we have $I=Ann_R(0:_MI)$ \cite{Fa95}. Following \cite[Definition 2.7]{AF11}, let
$$
I^M_0(M)= \cap \{L \mid  L \\\ is \\\ a \\\ completely \\\
irreducible
\\\ submodule \\\ of \\\ M\\\ and
$$
$$
 rM\subseteq L \\\ for \\\ some \\\ 0\not=r \in R \}.
$$

A commutative ring $R$ is  said to be a \textit{morphic ring} if, for
every element $a \in R$, $R/Ra \cong Ann_R(Ra)$ (isomorphism of $R$-modules). $R$ is morphic if and only if, for every $a \in R$, there exists $b \in R$ such that
$Ra = Ann_R(Rb)$ and $Ann_R(Ra) = Rb$ \cite[Lemma1]{MR2022487}.
Nicholson and S´anchez Campos extended the notion of morphic rings to
modules (\cite{NS05}). An endomorphism $a$ of a module $M$ is called \textit{morphic} if $M/a(M)\cong Ker(a)$. The module $M$ is called a \textit{morphic module} if every endomorphism of $M$ is morphic. Note that the endomorphisms of the $R$-module
$R$ are only the multiplication morphisms. Thus a ring $R$ is morphic if and only
if $R$ is a morphic $R$-module.
Recently in \cite{BTTK2022}, the authors introduced and investigated a new definition of morphic $R$-module which is also an extension of the notion of morphic rings.
An $R$-module $M$ is said to be a \textit{morphic $R$-module} if, for each $m \in M$, there exists
$a \in R$ such that $Rm = (0:_Ma)$ and $Ann_R(Rm) = Ra + Ann_R(M)$ \cite{BTTK2022}.

In this paper, we introduce and study the dual notion of morphic modules over a commutative ring in the sense of \cite{BTTK2022}.
We say that an $R$-module $M$ is a \textit{comorphic module} if for each completely irreducible submodule $L$ of $M$,
there exists $a \in R$ such that $L = aM$ and $(L:_RM) = Ra + Ann_R(M)$ (Definition \ref{1}).
 Example \ref{9}, shows that the concepts of morphic and comorphic $R$-modules are different in general. Among other results, it is shown that $M$ is a comorphic $R$-module if and only if for every submodule $N$ of $M$ with $M/N$ is a finitely cogenerated $R$-module, there exists $a \in R$ such that
$N = aM$ and $(N:_RM)= (a) + Ann_R(M)$ (Theorem \ref{5}). Also, it is proved that if $M$ is a comorphic $R$-module,  then the finite sum of  completely irreducible submodules of $M$ is a completely irreducible submodule of $M$ (Corollary \ref{4}). Proposition \ref{19}, shows that  every comorphic $R$-module is co-Hopfian. We say that an $R$-module $M$ is a \textit{co-Bézout module} if evey submodule $N$ of $M$ with $M/N$ is finitely cogenerated is a completely irreducible submodule of $M$ (Definition \ref{11}). It is shown that
if $M$ is a comorphic $R$-module which satisfies the $DAC$ and $I^M_0(M)\not=0$, then $R$ is a morphic ring and $M$ is a co-Bézout module (Theorem \ref{10}). Finally, the notion of quasi morphic (resp. quasi comorphic) $R$-module which is a generalization of morphic ring is difined (Definition \ref{20}). Morover, we show that if $M$ is cyclic (resp. a cocyclic $R$-module which satisfies the $DAC$) and morphic $R$-module in the sense of Nicholson and Sánchez Campos, then $M$ is a quasi morphic (resp. quasi comorphic) $R$-module (Proposition \ref{21}).
 It should be note that most of the results in this paper are the dual of the results in the paper \cite{BTTK2022}.
%%%%%%%%%%%%%%%%%%%%%%%%%
%%%%%%%%%%%%%%%%%%%%%%%%%%%%%%%%%
%%%%%%%%%%%%%%%%%%%%%%%%%%%%%%%%%%%%%%%%
%%%%%%%%%%%%%%%%%%%%%%%%%%%%%%%%%%%%%%%%%%%%%%%%
\section{Main results}
We frequently use the following basic facts.
\begin{rem}\label{r2.1}
Let $M$ be an $R$-module.
\begin{itemize}
\item [(a)] If $L$ is a completely irreducible submodule of $M$, then $(L:_Ma)$ is a completely irreducible submodule of $M$ \cite[Lemma 2.1]{MR3588217}.
\item [(b)] If $N$ and $K$ are two submodules of $M$, to prove $N\subseteq K$, it is enough to show that if $L$ is a completely irreducible submodule of $M$ such that $K\subseteq L$, then $N\subseteq L$ \cite[Remark 2.1]{AF1001}.
\end{itemize}
 \end{rem}

\begin{defn}\label{1}
We say that an $R$-module $M$ is a \textit{comorphic module} if for each completely irreducible submodule $L$ of $M$,
there exists $a \in R$ such that $L = aM$ and $(L:_RM) = Ra + Ann_R(M)$. This can be regarded as a dual notion of morphic module in the sense of \cite{BTTK2022}.
\end{defn}

A submodule $N$ of an $R$-module $M$ is said to be a \textit{multiple} of $M$, provided that $N = rM$ for some $r \in R$. If every submodule of $M$ is a multiple of $M$, then $M$ is said to be a \textit{principal ideal multiplication module} \cite{MR2451996}.

\begin{ex}\label{2}
By using \cite[Corollary of Theorem 9]{MR933916}, every finitely generated principal ideal multiplication module over a principal ideal ring is a comorphic module. In particular, the
$\Bbb Z$-module $\Bbb Z_n$ for each positive integer $n\geq 2$ is a comorphic  $\Bbb Z$-module. Also,  every principal ideal ring $R$ is a comorphic $R$-module.
\end{ex}

\begin{prop}\label{3}
Let $M$ be a comorphic $R$-module. Then the sum of
two completely irreducible submodules of $M$ is a completely irreducible submodule of $M$.
\end{prop}
\begin{proof}
Let $L_1$ and $L_2$ be two completely irreducible submodules of $M$. As $M$ is a comorphic $R$-module, there exist $a , b \in R$ such that $L_1=aM$, $(L_1:_RM) = Ra + Ann_R(M)$,  $L_2=bM$ , and $(L_2:_RM) = Rb + Ann_R(M)$. Thus by assumption,
 $(L_2:_Ma)=cM$, $((L_2:_Ma):_RM) = Rc + Ann_R(M)$ for some $c \in R$. We will show that $L_1+L_2=(L_2:_Mc)$. Let $x \in L_1+L_2=aM+bM$. Then $cx \in caM+cbM\subseteq L_2$. This implies that $L_1+L_2\subseteq (L_2:_Mc)$.
On the other hand,
$(L_2:_Ma)=cM$ implies that $caM \subseteq L_2$. Thus $L_1=aM \subseteq (L_2:_Mc)$. Therefore, $L_1+L_2 \subseteq  (L_2:_Mc)$ and by Remark \ref{r2.1} (a), we are done.
\end{proof}

\begin{cor}\label{4}
If $M$ is a comorphic $R$-module,  then the finite sum of
completely irreducible submodules of $M$ is a completely irreducible submodule of $M$.
\end{cor}
\begin{proof}
This follows from Proposition \ref{3}.
\end{proof}

The following theorem, is a characterization for the comorphic $R$-modules.
\begin{thm}\label{5}
For an $R$-module $M$ the following are equivalent:
\begin{itemize}
\item [(a)] $M$ is a comorphic $R$-module;
\item [(b)] For every submodule $N$ of $M$ with $M/N$ is a finitely cogenerated $R$-module, there exists $a \in R$ such that
$N = aM$ and $(N:_RM)= Ra + Ann_R(M)$.
\end{itemize}
\end{thm}
\begin{proof}
$(a) \Rightarrow (b)$
Let $N$ be a submodule of $M$ with $M/N$ be a finitely cogenerated $R$-module. Then  there exist completely irreducible submodules $L_1, L_2,\ldots ,L_n$ of $M$ such that $N=\cap_{i=1}^nL_i$. We proceed by induction on $n$. The induction start is just the definition
of comorphic modules. If $n\geq  1$, let $K =\cap^{n-1}_{i=1}L_i$. The induction hypothesis implies that
$K = aM$ and $(K:_RM)= Ra + Ann_R(M)$ for some $a \in R$. By part (a), $L_n = bM$ and $(L_n:_RM)= Rb + Ann_R(M)$  for some $b \in R$.  Also, by part (a), $(L_n :_Ma)= cM$ and $((L_n :_Ma):_RM)= Rc + Ann_R(M)$  for some $c \in R$. We claim that $N = acM$. Let $x \in  acM$. Then $x \in  acM\subseteq aM=K$ and   $x \in  acM\subseteq L_n$. Thus $x \in K \cap L_n=N$ and so $acM \subseteq N$. Now let $x \in N=K \cap L_n=aM \cap bM$. Then $x=am_1=bm_2$ for some $m_1,m_2 \in M$. Thus $x=am_1 \in L_n$ implies that $m_1 \in (L_n:_Ma)=cM$. Hence $m_1=cm_3$ for some $m_3 \in M$.  It follows that $x=acm_3\in acM$. Therefore, $N \subseteq acM$. Now we claim that $(N :_RM)= Rac + Ann_R(M)$.
Let $t \in (N:_RM)=(K\cap L_n:_RM)$. Then $t \in (K:_RM)$ and so $t=sa+u$ for some $s \in R$ and $u \in Ann_R(M)$ and $tM \subseteq L_n$. Thus $saM=tM\subseteq L_n$ and so $s \in ((L_n:_Ma):_RM)= Rc + Ann_R(M)$. Hence $s=hc+w$ for some $h \in R$ and $w \in Ann_R(M)$. Hence,
$t=cha+aw+u \in  Rac + Ann_R(M)$.  Consequently, $(N:_RM) \subseteq  Rac + Ann_R(M)$. Now let $t \in Rac + Ann_R(M)$. Then $t=acr+u_1$ for some $r \in R$ and  $u_1 \in Ann_R(M)$. It follows that  $tM=acrM=c(arM)\subseteq cM=(L_n:_Ra)$ and so $ta \in (L_n:_RM)$. Also, $tM=acrM \subseteq aM=K$. Therefore, $t \in (N:_RM)$. Hence, $Rac + Ann_R(M)\subseteq (N:_RM)$.
Thus we are done.

$(b) \Rightarrow (a)$
This is clear.
\end{proof}

\begin{cor}\label{6}
Every Artinian comorphic  $R$-module $M$ is a principal ideal multiplication $R$-module.
\end{cor}
\begin{proof}
This follows from Theorem \ref{5}.
\end{proof}

\begin{lem}\label{7}
Let $f : M\rightarrow \acute{M}$ be an  $R$-homomorphism. Then we have the following.
\begin{itemize}
\item [(a)] If $\{N_i\}_{i \in I}$ is a family of submodules of $M$, then
$$
f(\cap_{i\in I}(N_i+Ker f))=\cap_{i\in I}f(N_i+Ker f)=\cap_{i\in I}f(N_i).
$$
\item [(b)] If $f$ is an epimorphism and $\acute{L}$ is a completely irreducible submodule of $\acute{M}$, then $f^{-1}(\acute{L})$ is a  completely irreducible submodule of $M$.
\end{itemize}
\end{lem}
\begin{proof}
(a) Let  $\{N_i\}_{i \in I}$ be a family of submodules of $M$. Clearly, $f(\cap_{i\in I}(N_i+Ker f))\subseteq \cap_{i\in I}f(N_i+Ker f)$ and $\cap_{i\in I}f(N_i+Ker f)=\cap_{i\in I}f(N_i)$. Now let $x \in  \cap_{i\in I}f(N_i)$. Then for each $i \in I$ we have $x=f(n_i)$ for some $n_i \in N_i$. Now for each $j \in I$, we have $x=f(n_i)=f(n_j)$. This implies that $n_i-n_j \in Ker f$ and so $n_i+Ker f=n_j+Ker f$. It follows that
 $n_i+Ker f \in \cap_{i\in I}(N_i+Ker f)$. Thus $x=f(n_i+Ker f) \in f(\cap_{i\in I}(N_i+Ker f))$,
as needed.

(b) Let $f$ be an epimorphism and $\acute{L}$ be a completely irreducible submodule of $\acute{M}$. Assume that $\{N_i\}_{i \in I}$ is a collection of submodules of $M$ such that $f^{-1}(\acute{L})=\cap_{i \in I}N_i$. Clearly, $Ker f \subseteq f^{-1}(\acute{L}) \subseteq N_i$ for each $i \in I$. Then as $f$ is epic and part (a), we have
$\acute{L}=f(f^{-1}(\acute{L}))=f(\cap_{i \in I}N_i)=f(\cap_{i \in I}(N_i+Ker f))=\cap_{i \in I}f(N_i)$. Now $\acute{L}$ is a completely irreducible submodule of $\acute{M}$ implies that $\acute{L}=f(N_j)$ for some $j \in I$. Thus $f^{-1}(\acute{L})=f^{-1}f(N_j)=N_j+Ker f=N_j$, as required.
\end{proof}

\begin{prop}\label{8}
Let $f : M\rightarrow \acute{M}$ be an $R$-epimorphism. If $M$ is a
comorphic $R$-module, then so is $\acute{M}$.
\end{prop}
\begin{proof}
Let $M$ be a comorphic $R$-module and $\acute{L}$ be a completely irreducible submodule of $\acute{M}$. Then by Lemma \ref{7}(b),
$f^{-1}(\acute{L})$ is a completely irreducible submodule of $M$. Thus by assumption, there exists $a \in R$ such that
$f^{-1}(\acute{L})= aM$ and $(f^{-1}(\acute{L}):_RM)= Ra + Ann_R(M)$.
On the other hand if $t \acute{M}=0$, then $tM \subseteq Ker(f) \subseteq f^{-1}(\acute{L})$. Thus $ t \in Ra+ Ann_R(M)$ and so $Ann_R(\acute{M}) \subseteq Ra+ Ann_R(M)$. This implies that $Ra + Ann_R(M)=Ra+ Ann_R(\acute{M})$.
Now as $f$ is an epimorphism,  $\acute{L}=af(M)=a\acute{M}$ and $(\acute{L}:_R\acute{M})= Ra+Ann_R(\acute{M})$.
\end{proof}

A non-zero submodule $S$ of an $R$-module $M$ is said to be \emph{second} if for each $a \in R$, the homomorphism $ S \stackrel {a} \rightarrow S$  which is defined by multiplication by $a$, is either surjective or zero \cite{Y01}.
\begin{prop}\label{8}
An $R$-module $M$ is a second comorphic module if and only if $M$ is a simple $R$-module.
\end{prop}
\begin{proof}
If $M$ is a simple $R$-module, then clearly $M$ is a second comorphic $R$-module. Conversely, let $M$ be a second comorphic $R$-module and $L$ be a completely irreducible submodule of $M$. Then there exists $a \in R$ such that $L=aM$ and $(L:_RM)= Ra +Ann_R(M)$. As $M$ is second, $aM=0$ or $aM=M$. Thus $M$ is a simple $R$-module.
\end{proof}

The following example is illustrating that the notions of morphic and comorphic for modules do not  generally coincide.
\begin{ex}\label{9}
For a prime number $p$, the $\Bbb Z$-module $\Bbb Z_{p^\infty}$ is not a comorphic $\Bbb Z$-module, but it is a morphic $\Bbb Z$-module. Also, the $\Bbb Z$-module $\Bbb Z$ is a comorphic $\Bbb Z$-module, but it is not a morphic $\Bbb Z$-module.
\end{ex}

\begin{defn}\label{11}
We say that a submodule $N$ of an $R$-module $M$ is a \textit{co-Baer submodule} if for each completely irreducible submodule $L$ of $M$ with $N\subseteq L$, we have $N\subseteq (L:_RM)M$.
\end{defn}

\begin{prop}\label{12}
Let $M$ be a comorphic $R$-module. Then we have the following.
\begin{itemize}
\item [(a)] There exists $a \in R$ such that $N=(N:_RM)M=aM$ for each submodule $N$ of $M$ with $M/N$ is finitely cogenerated  $R$-module.
\item [(b)] Every submodule $N$ of $M$ is a co-Baer submodule.
\end{itemize}
\end{prop}
\begin{proof}
(a) Suppose that $N$ is a submodule of $M$ with $M/N$ is finitely cogenerated $R$-module. By Theorem \ref{5},
$N = aM$ and $(N:_RM)=Ra + Ann_R(M)$ for some $a \in R$. Thus $N = aM$ and $(N:_RM)M=aM$. Hence, $N=(N:_RM)M$.

(b) Let $L$ be a completely irreducible submodule of $M$ with $N\subseteq L$. Since $L$ is completely irreducible submodule of $M$,
$L=(L:_RM)M$ by part (a). Thus $N\subseteq (L:_RM)M$.
\end{proof}

\begin{defn}\label{11}
We say that an $R$-module $M$ is a \textit{co-Bézout module} if every submodule $N$ of $M$ with $M/N$ is finitely cogenerated is a completely irreducible submodule of $M$.
\end{defn}

\begin{thm}\label{10}
Let $M$ be a comorphic $R$-module which satisfies the $DAC$ and $I^M_0(M)\not=0$. Then
\begin{itemize}
\item [(a)] $R$ is a morphic ring.
\item [(b)] $M$ is a co-Bézout module.
\end{itemize}
\end{thm}
\begin{proof}
(a) As $I^M_0(M)\not=0$, there exists a completely irreducible submodule $L$ of $M$ such that $I^M_0(M)\not\subseteq L$ and $Ann_R(M)=0$. Let $a \in R$. Since $M$ is a comorphic $R$-module, there exists $b \in R$ such that
$(L:_Ma)=bM$ and $((L:_Ma):_RM)=Rb+Ann_R(M)=Rb$. As $I^M_0(M)\not=0$, $((L:_Ma):_RM)=Ann_R(Ra)$. Thus $Ann_R(Ra)=Rb$. On the other hand $(L:_Ma)=bM$ implies that $Ann_R((L:_Ma))=Ann_R(Rb)$. Hence, $Ann_R(Rb)\subseteq Ann_R((0:_MAnn_R(L)a)=Ann_R(L)a\subseteq Ra$. We have $(L:_Ma)=bM$ follows that $abM \subseteq L$. Since  $I^M_0(M)\not\subseteq L$, we have $ab=0$ and so $Ra \subseteq Ann_R(Rb)$. Therefore, $R$ is a morphic ring.

(b) Let $N$ be a submodule of $M$ with $M/N$ be a finitely cogenerated $R$-module and $L$ be a completely irreducible submodule of $M$ with $I^M_0(M)\not\subseteq L$. Since
$Ann_R(M) = 0$, by Theorem \ref{5}, $N = aM$ and $(N:_RM)=Ra$ for some $a \in R$. By part (a), $R$ is a morphic ring and so
there exists $b\in R$ such that $Ra = Ann_R(Rb)$ and $Rb = Ann_R(Ra)$.  Since $I^M_0(M)\not\subseteq L$, we have
$Ann_R(Rb)=((L:_Rb):_RM)$.
By Proposition \ref{12}, $((L:_Mb):_RM)M=(L:_Mb)$ and $N=(N:_RM)M$. Thus
$$
N=aM=Ann_R(Rb)M=((L:_Mb):_RM)M=(L:_Mb).
$$
Therefore, $M$ is a co-Bézout $R$-module by Remark \ref{r2.1} (a).
\end{proof}

\begin{thm}\label{14}
Let $M$ be a comorphic $R$-module. Consider the following conditions:
\begin{itemize}
\item [(a)] $M$ satisfies the $ACC$ on completely irreducible submodules.
\item [(b)] $R$ satisfies the $ACC$ on $\{(L:_RM) :L\ is \ a \ completely\ irreducible\ submodule \ of\ M\}$.
\item [(c)] $M$ satisfies the $ACC$ on submodules $N$ with $M/N$ is finitely cogenerated.
\item [(d)] $R$ satisfies the $ACC$ on $\{(N:_RM) :M/N\ is \ a \ finitely\ cogenerated\ module\}$.
\end{itemize}
Then, $(a) \Leftrightarrow (b) \Leftarrow (c) \Leftrightarrow (d)$. Furthermore, if $M$ satisfies the $DAC$ and $I^M_0(M)\not=0$, then
$(a) \Leftrightarrow (b) \Leftrightarrow (c) \Leftrightarrow (d)$.
\end{thm}
\begin{proof}
$(a)\Rightarrow (b)$
Consider the following ascending chain of ideals of $R$
$$
(L_1:_RM) \subseteq (L_2:_RM)  \subseteq \cdots \subseteq (L_n:_RM)  \subseteq \cdots,
$$
where $L_i$ is a completely irreducible submodule of $M$. By Proposition \ref{12}, $L_i= (L_i:_RM)M$ for each $i$. Thus we
have the following ascending chain of completely irreducible submodules of $M$
$$
L_1 \subseteq L_2\subseteq \cdots \subseteq L_n\subseteq \cdots .
$$
Thus, by hypothesis, there exists an integer $k \geq 1$ such that $L_k= L_{k+n}$ for each
$n \geq 0$. This implies that $(L_k:_RM) =(L_{n+k}:_RM)$ for each
$n \geq 0$, as needed.

$(b) \Rightarrow (a)$
This is similar to $(a) \Rightarrow (b)$.

$(a) \Leftrightarrow (d)$ This is similar to $(a) \Leftrightarrow (b)$) (by using Proposition \ref{12}).

$(c) \Rightarrow (a)$ This is clear.

If $M$ satisfies the $DAC $ and $I^M_0(M)\not=0$, then the result follows from Theorem \ref{10} (b).
\end{proof}

\begin{lem}\label{15}
Let $M$ be a finitely generated and finitely cogenerated $R$-module such that $(0:_Ma)=0$ for some $a \in R$. Then there exists $t \in R$ such that $(1-ta)M=0$.
\end{lem}
\begin{proof}
Set $S:= 1 +Ra$. Clearly, $S$ is
a multiplicatively closed subset of $R$ and $S^{-1}(Ra)$ is a unique maximal ideal of $S^{-1}R$ by \cite[Exe. 2, page 43]{MR3525784}.  We have $(0 :_{S^{-1}(M)} S^{-1}(Ra)) = 0$.  Since $S^{-1}(M)$ is finitely cogenerated $S^{-1}R$-module $S^{-1}(M) = 0$ by \cite[Theorem 3.14]{AF08}. Now since $M$ is finitely generated, there exists an element $t \in R$
such that $(1-ta)(M) = 0$ by \cite[Exe. 1, page 43]{MR3525784}.
\end{proof}

\begin{thm}\label{16}
Let $M$ be a finitely generated $R$-module such that for each submodule $N$ of $M$ with $M/N$ is a finitely cogenerated $R$-module there exists $a \in R$ such that $N=(0:_Ma)=(0:_Ma^2)$. Then $M$ is a comorphic $R$-module.
\end{thm}
\begin{proof}
Let $N$ be a submodule of $M$ with $M/N$ be a finitely cogenerated $R$-module and $N=(0:_Ma)=(0:_Ma^2)$ for some $a \in R$.
Then $(0:_{M/N}a)=0$.  Thus by Lemma \ref{15}, there exists an element $t \in R$
such that $(1-ta) \in (N:_RM)$. It follows that $R(1-ta)+Ann_R(M)\subseteq (N:_RM)$.
Now, let $b \in  (N:_RM)$. Then, $bM \subseteq N$ and so
$tab \in Ann_R(M)$. Thus $b = b(1 -ta)+tba \in R(1-ta) + Ann_R(M)$.
Hence,  $(N:_RM)=R(1-ta)+Ann_R(M)$. Let $L$ be a completely irreducible submodule of $M$ such that $(1 - ta)M \subseteq L$
and $y \in N$. Then $(1-ta)y \in L$ and $ya=0$. Thus $y \in L$ and so $N\subseteq L$. Therefore $N \subseteq (1 - ta)M $ by Remark \ref{r2.1} (b). Hence $N = (1 - ta)M $, as needed.
\end{proof}

\begin{prop}\label{17}
Let $M$ be a finite length comorphic $R$-module and $S$ be a multiplicatively closed subset
of $R$. Then $S^{-1}M$ is a comorphic  $S^{-1}R$-module.
\end{prop}
\begin{proof}
First we note that as $M$ is finite length,  $S^{-1}M$  is finite length.
Let $S^{-1}N$ be a submodule of $S^{-1}M$. Since $M$ is comorphic, $N=aM$ and $(N:_RM)=Ra+Ann_R(M)$ for some $a \in R$. This implies that  $S^{-1}N=(a/1)S^{-1}M$ and $(S^{-1}N:_{S^{-1}R}S^{-1}M)=(S^{-1}R)(a/1)+Ann_{S^{-1}R}(S^{-1}M)$, as needed.
\end{proof}

Let $R_i$ be a commutative ring with identity, $M_i$ be an $R_i$-module for each $i = 1, 2$. Assume that
$M = M_1\times M_2$ and $R = R_1\times R_2$. Then $M$ is clearly
an $R$-module with componentwise addition and scalar multiplication. Each submodule $N$ of $M$ is of the form $N = N_1\times N_2$, where $N_i$ is a submodule of $M_i$ for each $i = 1, 2$.

\begin{prop}\label{18}
Let $R = R_1 \times R_2$, where $R_i$ is a commutative ring
with identity and  let $M = M_1 \times M_2$
be an $R$-module, where $M_1$ is an $R_1$-module and $M_2$ is an $R_2$-module. Then the following conditions are equivalent:
\begin{itemize}
  \item [(a)] $M$ is a comorphic $R$-module;
  \item [(b)] $M_i$ is a comorphic $R_i$-module for each $i = 1, 2$.
\end{itemize}
\end{prop}
\begin{proof}
$(a)\Rightarrow (b)$
Let $N_1$ be a submodule of $M_1$
with $M_1/N_1$ is a finitely cogenerated $R_1$-module. Then $N_1\times 0$ is a submodule of $M_1\times M_2$ with $M/(N_1 \times 0)$ is a finitely cogenerated $R$-module. Now by part (a), $N_1 \times 0=aM$, where $a=(a_1,a_2)\in R$ and $(N_1\times 0:_RM)=R(a_1, a_2)+Ann_R(M)$.
This implies that $N_1=a_1M_1$ and $(N_1:_{R_1}M_1)=R_1a_1+Ann_{R_1}(M_1)$. Hence $M_1$ is a comorphic $R_1$-module. Similar argument shows that $M_2$ is a comorphic $R_2$-module.

$(b)\Rightarrow (a)$
Let $N=N_1 \times N_2$ be a submodule of $M = M_1 \times M_2$
with $M/N$ be a finitely cogenerated $R$-module. Then $N_i$ is a submodule of $M_i$ with $M_i/N_i$ is a finitely cogenerated $R_i$-module for each $i = 1, 2$. Thus by part (b), $N_i=a_iM_i$ and $(N_i:_{R_i}M_i)=R_ia_i+Ann_{R_i}(M_i)$ for each $i = 1, 2$.
Therefore, $N = N_1 \times N_2=(a_1, a_2)(M_1 \times M_2)$ and $(N_1 \times N_2:_{R_1 \times R_2}M_1 \times M_2)=(R_1 \times R_2)(a_1, a_2)+Ann_{R_1 \times R_2}(M_1 \times M_2)$ as required.
\end{proof}

Recall that an $R$-module $M$ is said to be \textit{Hopfian} (resp. \textit{co-Hopfian})
if every surjective endomorphism (resp.  injective endomorphism) $f$ of $M$  is an isomorphism \cite{Hir} (resp. \cite{Vara}).
\begin{prop}\label{19}
\begin{itemize}
  \item [(a)] Every comorphic $R$-module is Hopfian.
  \item [(b)] Every morphic $R$-module is co-Hopfian.
\end{itemize}
\end{prop}
\begin{proof}
(a) Let $M$ be a comorphic $R$-module and $f:M\rightarrow M$ be a surjective endomorphism. Let $L$ be a completely irreducible submodule of $M$. By Lemma \ref{7} (b), $f^{-1}(L)$ is a completely irreducible submodule of $M$. Thus by assumption, $Ker(f)\subseteq f^{-1}(L)=aM$ for some $a \in R$. Since $f$ is surjective, $L=f( f^{-1}(L))=aM$. Therefore, $Ker(f)\subseteq L$. Thus by Remark \ref{r2.1} (b), $Ker(f)=0$ and we are done.

(b) Let $M$ be a morphic $R$-module and $f:M\rightarrow M$ be an injective endomorphism. Let $m \in M$. Then by assumption, $f(m)\in f(m)R=(0:_Ma)$ for some $a \in R$. Since $f$ is injective, $am=0$. Therefore, $ m \in (0:_Ma)=f(m)R\subseteq Im (f)$. Thus $M=Im (f)$.
\end{proof}

\begin{defn}\label{20}
\begin{itemize}
  \item [(a)] We say that an $R$-module $M$ is a \textit{quasi morphic module} if for each $a \in R$, we have $(0:_Ma)$ is a cyclic submodule of $M$ and $Ann_R((0:_Ma))=Ra+Ann_R(M)$.
  \item [(b)] We say that an $R$-module $M$ is a \textit{quasi comorphic module} if for each $a \in R$, we have $aM$ is a completely irreducible submodule of $M$ and $(aM:_RM)=Ra+Ann_R(M)$.
\end{itemize}
\end{defn}

It is easy to see that $R_R$ is quasi morphic if and only if $R$ is a morphic ring.

Recall that an $R$-module $M$ is said to be \emph{cocyclic} if
$M$ has a simple essential socle \cite{Yassemi1998T}.
\begin{prop}\label{21}
\begin{itemize}
  \item [(a)] Let $M$ be a cyclic $R$-module. If $M$ is a morphic $R$-module in the sense of Nicholson and Sánchez Campos, then $M$ is a quasi morphic $R$-module.
   \item [(b)] Let $M$ be a cocyclic $R$-module which satisfies the $DAC$. If $M$ is a morphic $R$-module in the sense of Nicholson and Sánchez Campos, then $M$ is a quasi comorphic $R$-module.
       \end{itemize}
\end{prop}
\begin{proof}
(a) Let $M$ be a morphic $R$-module in the sense of Nicholson and Sánchez Campos and let $a \in R$. Then $M/aM\cong (0:_Ma)$.  Thus $(aM:_RM)=Ann_R(M/aM)=Ann_R((0:_Ma))$. As $M$ is cyclic, $(aM:_RM)=Ra+Ann_R(M)$. Thus   $Ann_R((0:_Ma))=Ra+Ann_R(M)$. Since, $M$ is cyclic, $M/aM$ is cyclic. Thus  $(0:_Ma)$ is cyclic.

(b) Let $M$ be a morphic $R$-module in the sense of Nicholson and Sánchez Campos and let $a \in R$. Then $M/aM\cong (0:_Ma)$ and so $(aM:_RM)=Ann_R((0:_Ma))$. As $M$ satisfies the $DAC$, we have $Ann_R((0:_Ma))=Ra+Ann_R(M)$. Thus $(aM:_RM)=Ra+Ann_R(M)$.  Since, $M$ is cocyclic,  $(0:_Ma)$ is cocyclic by \cite[Lemma 1.2]{Yassemi1998T}. Thus $M/aM$ is cocyclic and so $aM$ is a completely irreducible submodule of $M$ by \cite[Remark 1.1]{FHo06}.
\end{proof}

\end{document}